\documentclass[12pt]{article}
\usepackage{amsmath,amsthm,amssymb}
\usepackage{amssymb,latexsym}

\newtheorem{theorem}{Theorem}

\newtheorem{lemma}{Lemma}

\begin{document}
\baselineskip=17pt

\title{\bf Consecutive square-free numbers of a special form}

\author{\bf S. I. Dimitrov}

\date{}
\maketitle

\begin{abstract}

In the present paper we prove that for any fixed $1<c<7/6$ there exist infinitely many
consecutive square-free numbers of the form $[n^c], [n^c]+1$
and we also establish an asymptotic formula in given interval.
\medskip

\end{abstract}

\section{Notations}
\indent

Let $X$ be a sufficiently large positive number. By $\varepsilon$ we denote an arbitrary small positive number,
not necessarily the same in different occurrences. We denote by $\mu(n)$ the M\"{o}bius function and by $\tau(n)$
the number of positive divisors of $n$. As usual $[t]$ and $\{t\}$ denote the integer part, respectively, the
fractional part of $t$. Let $||t||$ be the distance from $t$ to the nearest integer.
Instead of $m\equiv n\,\pmod {k}$ we write for simplicity $m\equiv n\,(k)$.
As usual $e(t)$=exp($2\pi it$). For positive $A$ and $B$ we write $A\asymp B$ instead of $B\ll A\ll B$.
Let $c$ be a real constant such that $1<c<7/6$.

Denote
\begin{align}
\label{z}
&z=X^{\frac{2c-1}{4}}\,;\\
\label{gamma}
&\gamma=1/c\,;\\
\label{psi}
&\psi(t)=\{t\}-1/2\,;\\
\label{sigma}
&\sigma=\prod\limits_{p}\left(1-\frac{2}{p^2}\right)\,.
\end{align}
\newpage
\section{Introduction and statement of the result}
\indent

The problem for the consecutive square-free numbers arises in 1932 when Carlitz \cite{Carlitz} proved that
\begin{equation}\label{Carlitz}
\sum\limits_{n\leq X}\mu^2(n)\mu^2(n+1)=\sigma X+\mathcal{O}\big(X^{\theta+\varepsilon}\big)\,,
\end{equation}
where $\theta=2/3$ and $\sigma$ is denoted by \eqref{sigma}.
Formula \eqref{Carlitz} was subsequently improved by Heath-Brown
\cite{Heath-Brown} to $\theta=7/11$ and by Reuss \cite{Reuss} to $\theta=(26+\sqrt{433})/81$.

On the other hand in 2008 Cao and Zhai \cite{Cao-Zhai2} proved that
for any fixed $1<c<149/87$  the asymptotic formula
\begin{equation}\label{Cao-Zhai}
\sum\limits_{n\leq X}\mu^2([n^c])=\frac{6}{\pi^2}X+\mathcal{O}\left(X^{1-\varepsilon}\right)
\end{equation}
holds.
Their earlier result  \cite{Cao-Zhai1} covers the narrower range $1<c<61/36$.

Define
\begin{equation}\label{Sc}
S_c(X)=\sum\limits_{X/2<n\leq X}\mu^2([n^c])\mu^2([n^c]+1)\,.
\end{equation}

We couple the theorems \eqref{Carlitz} and \eqref{Cao-Zhai} by proving
\begin{theorem}Let $1<c<7/6$. Then the asymptotic formula
\begin{equation*}
S_c(X)=\frac{1}{2}\sigma X+\mathcal{O}\left(X^{\frac{6c+1}{8}+\varepsilon}\right)\,,
\end{equation*}
holds. Here $\sigma$ is defined by \eqref{sigma}.
\end{theorem}

\section{Lemmas}
\indent

\begin{lemma}\label{Korput}
Suppose that $f''(t)$ exists, is continuous on $[a,b]$ and satisfies
\begin{equation*}
f''(t)\asymp\lambda\quad(\lambda>0)\quad\mbox{for}\quad t\in[a,b]\,.
\end{equation*}
Then
\begin{equation*}
\bigg|\sum_{a<n\le b}e(f(n))\bigg|\ll(b-a)\lambda^{1/2}+\lambda^{-1/2}\,.
\end{equation*}
\end{lemma}
\begin{proof}
See (\cite{Titchmarsh}, Ch.5, Th.5.9).
\end{proof}

\begin{lemma}\label{tauest}
Let $n\in \mathbb{N}$. Then
\begin{equation*}
\tau(n)\ll n^\varepsilon\,.
\end{equation*}
\end{lemma}

\section{Proof of the Theorem}
\indent

We use \eqref{Sc} and the well-known identity $\mu^2(n)=\sum_{d^2|n}\mu(d)$ to write

\begin{align}\label{Scest1}
S_c(X)&=\sum\limits_{X/2<n\leq X}\sum\limits_{d^2|[n^c]}\mu(d)\sum\limits_{t^2|[n^c]+1}\mu(t)\nonumber\\
&=\sum\limits_{d,t\atop{(d,t)=1}}\mu(d)\mu(t)\sum\limits_{X/2<n\leq X\atop{[n^c]\equiv0\,(d^2)
\atop{[n^c]+1\equiv0\,(t^2)}}}1\nonumber\\
&=\sum\limits_{d,t\atop{(d,t)=1}}\mu(d)\mu(t)\sum\limits_{((X/2)^c-1)d^{-2}<k\leq X^cd^{-2}\atop{kd^2+1\equiv0\,(t^2)}}
\sum\limits_{X/2<n\leq X\atop{[n^c]=kd^2}}1\nonumber\\
&=\sum\limits_{d,t\atop{(d,t)=1}}\mu(d)\mu(t)\sum\limits_{((X/2)^c-1)d^{-2}<k\leq X^cd^{-2}\atop{kd^2+1\equiv0\,(t^2)}}
\sum\limits_{X/2<n\leq X\atop{kd^2\leq n^c<kd^2+1}}1\nonumber\\
&=\sum\limits_{d,t\atop{(d,t)=1}}\mu(d)\mu(t)\sum\limits_{(X/2)^cd^{-2}<k\leq X^cd^{-2}\atop{kd^2+1\equiv0\,(t^2)}}
\sum\limits_{(kd^2)^\gamma\leq n<(kd^2+1)^\gamma}1+\mathcal{O}(X^\varepsilon)\nonumber\\
&=S^{(1)}_c(X)+S^{(2)}_c(X)+\mathcal{O}(X^\varepsilon)\,,
\end{align}
where
\begin{equation}\label{Sc1}
S^{(1)}_c(X)=\sum\limits_{dt\leq z\atop{(d,t)=1}}\mu(d)\mu(t)
\sum\limits_{(X/2)^cd^{-2}<k\leq X^cd^{-2}\atop{kd^2+1\equiv0\,(t^2)}}
\sum\limits_{(kd^2)^\gamma\leq n<(kd^2+1)^\gamma}1\,,
\end{equation}
\begin{equation}\label{Sc2}
S^{(2)}_c(X)=\sum\limits_{dt>z\atop{(d,t)=1}}\mu(d)\mu(t)
\sum\limits_{(X/2)^cd^{-2}<k\leq X^cd^{-2}\atop{kd^2+1\equiv0\,(t^2)}}
\sum\limits_{(kd^2)^\gamma\leq n<(kd^2+1)^\gamma}1\,.
\end{equation}
\textbf{Estimation of} $\mathbf{S^{(1)}_c(X)}$

From \eqref{Sc1} we have
\begin{align}\label{Sc1est1}
S^{(1)}_c(X)&=\sum\limits_{dt\leq z\atop{(d,t)=1}}\mu(d)\mu(t)
\sum\limits_{(X/2)^cd^{-2}<k\leq X^cd^{-2}\atop{kd^2+1\equiv0\,(t^2)}}
\big([-(kd^2)^\gamma]-[-(kd^2+1)^\gamma]\big)\nonumber\\
&=S^{(3)}_c(X)+S^{(4)}_c(X)\,,
\end{align}
where
\begin{align}
\label{Sc3}
S^{(3)}_c(X)&=\sum\limits_{dt\leq z\atop{(d,t)=1}}\mu(d)\mu(t)
\sum\limits_{(X/2)^cd^{-2}<k\leq X^cd^{-2}\atop{kd^2+1\equiv0\,(t^2)}}
\Big((kd^2+1)^\gamma-(kd^2)^\gamma\Big)\,,\\
\label{Sc4}
S^{(4)}_c(X)&=\sum\limits_{dt\leq z\atop{(d,t)=1}}\mu(d)\mu(t)
\sum\limits_{(X/2)^cd^{-2}<k\leq X^cd^{-2}\atop{kd^2+1\equiv0\,(t^2)}}
\Big(\psi(-(kd^2+1)^\gamma)-\psi(-(kd^2)^\gamma)\Big)\,.
\end{align}

First we shall estimate $S^{(3)}_c(X)$.\\
Using \eqref{Sc3} and Abel's transformation we obtain
\begin{align}\label{Sc3est1}
S^{(3)}_c(X)&=\sum\limits_{dt\leq z\atop{(d,t)=1}}\mu(d)\mu(t)
\sum\limits_{(X/2)^cd^{-2}<k\leq X^cd^{-2}\atop{kd^2+1\equiv0\,(t^2)}}
(kd^2)^\gamma\Big(\gamma(kd^2)^{-1}+\mathcal{O}\big((kd^2)^{-2}\big)\Big)\nonumber\\
&=\gamma\sum\limits_{dt\leq z\atop{(d,t)=1}}\mu(d)\mu(t)d^{2(\gamma-1)}
\sum\limits_{(X/2)^cd^{-2}<k\leq X^cd^{-2}\atop{kd^2+1\equiv0\,(t^2)}}k^{\gamma-1}\nonumber\\
&+\mathcal{O}\Bigg(\sum\limits_{dt\leq z\atop{(d,t)=1}}d^{2(\gamma-2)}
\sum\limits_{(X/2)^cd^{-2}<k\leq X^cd^{-2}\atop{kd^2+1\equiv0\,(t^2)}}k^{\gamma-2}\Bigg)\nonumber\\
&=\gamma\sum\limits_{dt\leq z\atop{(d,t)=1}}\mu(d)\mu(t)d^{2(\gamma-1)}\Bigg[(X^cd^{-2})^{\gamma-1}
\sum\limits_{(X/2)^cd^{-2}<k\leq X^cd^{-2}\atop{kd^2+1\equiv0\,(t^2)}}1\nonumber\\
&-\int\limits_{(X/2)^cd^{-2}}^{X^cd^{-2}}\Bigg(\sum\limits_{(X/2)^cd^{-2}<k\leq y\atop{kd^2+1\equiv0\,(t^2)}}1\Bigg)
(y^{\gamma-1})'dy\Bigg]+\mathcal{O}\Bigg(X^{1-2c}\sum\limits_{dt\leq z\atop{(d,t)=1}}
\sum\limits_{(X/2)^cd^{-2}<k\leq X^cd^{-2}\atop{kd^2+1\equiv0\,(t^2)}}1\Bigg)\nonumber\\
&=\gamma\sum\limits_{dt\leq z\atop{(d,t)=1}}\mu(d)\mu(t)d^{2(\gamma-1)}\Bigg[X^{1-c}d^{2(1-\gamma)}
\Bigg(\frac{2^c-1}{2^c}\frac{X^c}{d^2t^2}+\mathcal{O}(1)\Bigg)\nonumber\\
&-\int\limits_{(X/2)^cd^{-2}}^{X^cd^{-2}}\Bigg(\frac{y-(X/2)^cd^{-2}}{t^2}+\mathcal{O}(1)\Bigg)
(y^{\gamma-1})'dy\Bigg]+\mathcal{O}\Bigg(X^{1-c}\sum\limits_{dt\leq z}\frac{1}{d^2t^2}\Bigg)\nonumber\\
&=\frac{1}{2}X\sum\limits_{dt\leq z\atop{(d,t)=1}}\frac{\mu(d)\mu(t)}{d^2t^2}
+\mathcal{O}\Bigg(X^{1-c}\sum\limits_{dt\leq z}1\Bigg)+\mathcal{O}(1)\nonumber\\
&=\frac{1}{2}X\sum\limits_{d,t=1\atop{(d,t)=1}}\frac{\mu(d)\mu(t)}{d^2t^2}
-\frac{1}{2}X\sum\limits_{dt>z\atop{(d,t)=1}}\frac{\mu(d)\mu(t)}{d^2t^2}
+\mathcal{O}\Bigg(X^{1-c}\sum\limits_{dt\leq z}1\Bigg)+\mathcal{O}(1)\,.
\end{align}
It is well-known that
\begin{equation}\label{sum1}
\sum\limits_{d,t=1\atop{(d,t)=1}}\frac{\mu(d)\mu(t)}{d^2t^2}=\prod\limits_{p}\left(1-\frac{2}{p^2}\right)\,,
\end{equation}
see (\cite{Tolev}, Theorem 2.6.8).\\
On the other hand by Lemma \ref{tauest}
\begin{equation}\label{sum2}
\sum\limits_{dt>z\atop{(d,t)=1}}\frac{\mu(d)\mu(t)}{d^2t^2}\ll\sum\limits_{dt>z}\frac{1}{d^2t^2}
=\sum\limits_{n>z}\frac{\tau(n)}{n^2}\ll\sum\limits_{n>z}\frac{1}{n^{2-\varepsilon}}\ll z^{\varepsilon-1}\,.
\end{equation}
By the same way
\begin{equation}\label{sum3}
\sum\limits_{dt\leq z}1=\sum\limits_{n\leq z}\tau(n)\ll zX^\varepsilon\,.
\end{equation}
Bearing in mind \eqref{z}, \eqref{sigma} and \eqref{Sc3est1} -- \eqref{sum3} we find
\begin{equation}\label{Sc3est2}
S^{(3)}_c(X)=\frac{1}{2}\sigma X+\mathcal{O}\left(X^{\frac{6c+1}{8}+\varepsilon}\right)\,.
\end{equation}

Now we shall estimate $S^{(4)}_c(X)$.\\
Replace
\begin{equation}\label{Phi}
\Phi(k,d)=\psi(-(kd^2+1)^\gamma)-\psi(-(kd^2)^\gamma)\,.
\end{equation}
Let $M\geq2$ is a real parameter, we shall choose latter depending on $X,d$ and $t$.
Using (\cite{Tolev}, Lemma 5.2.2) we get
\begin{equation}\label{Phisum}
\Phi(k,d)=-\frac{1}{2\pi i}\sum\limits_{1\leq|h|\leq M}\frac{\omega(k,d,h)}{h}
+\mathcal{O}\big(\omega_1(k,d)\big)+\mathcal{O}\big(\omega_2(k,d)\big)\,,
\end{equation}
where
\begin{equation}\label{omega}
\omega(k,d,h)=e(-h(kd^2+1)^\gamma)-e(-h(kd^2)^\gamma)\,,
\end{equation}
\begin{equation}\label{omega12}
\omega_1(k,d)=\min\left(1,\frac{1}{M||-(kd^2+1)^\gamma||}\right)\,,
\quad \omega_2(k,d)=\min\left(1,\frac{1}{M||-(kd^2)^\gamma||}\right)\,.
\end{equation}
From \eqref{Sc4}, \eqref{Phi}, \eqref{Phisum} and \eqref{omega12} it follows
\begin{equation}\label{Sc4est1}
S^{(4)}_c(X)=-\frac{1}{2\pi i}\sum\limits_{dt\leq z\atop{(d,t)=1}}\mu(d)\mu(t)
\sum\limits_{1\leq|h|\leq M}\frac{1}{h}
\sum\limits_{(X/2)^cd^{-2}<k\leq X^cd^{-2}\atop{kd^2+1\equiv0\,(t^2)}}\omega(k,d,h)
+\Omega_1+\Omega_2\,,
\end{equation}
where $\Omega_1$ and $\Omega_2$ are the contributions of the remainder terms in \eqref{Phisum}.

Using (\cite{Tolev}, Lemma 5.2.3) we obtain
\begin{align}\label{Omega12est1}
\Omega_1,\,\Omega_2&\ll\sum\limits_{dt\leq z\atop{(d,t)=1}}
\sum\limits_{(X/2)^cd^{-2}<k\leq X^cd^{-2}+1/d^2\atop{kd^2+1\equiv0\,(t^2)}}
\min\left(1,\frac{1}{M||(kd^2)^\gamma||}\right)\nonumber\\
&\ll\sum\limits_{dt\leq z\atop{(d,t)=1}}
\sum\limits_{(X/2)^cd^{-2}<k\leq X^cd^{-2}+1/d^2\atop{kd^2+1\equiv0\,(t^2)}}
\sum\limits_{h\in\mathbb{Z}}b_M(h)e(h(kd^2)^\gamma)\,,
\end{align}
where
\begin{equation}\label{bM}
b_M(h)\ll
\begin{cases}
\frac{\log M}{M}  &  \text{ if }\;    h\in\mathbb{Z} \,, \\
\frac{M}{h^2}  &  \text{ if }\;    h\in\mathbb{Z}\backslash\{0\} \,.
\end{cases}
\end{equation}
From \eqref{Omega12est1} and \eqref{bM} we find
\begin{equation}\label{Omega12est2}
\Omega_1,\,\Omega_2\ll\sum\limits_{dt\leq z}\frac{X^c}{d^2t^2}\frac{\log M}{M}
+\sum\limits_{dt\leq z}\frac{\log M}{M}\sum\limits_{1\leq|h|\leq M}|H(t,h)|
+\sum\limits_{dt\leq z}M\sum\limits_{|h|>M}\frac{|H(t,h)|}{h^2}\,,
\end{equation}
where
\begin{equation}\label{H}
H(t,h)=\sum\limits_{((X/2)^c+1)t^{-2}<l\leq(X^c+2)t^{-2}}e(h(lt^2-1)^\gamma)\,.
\end{equation}
Denote
\begin{equation*}
f(y)=h(yt^2-1)^\gamma\,.
\end{equation*}
We have
\begin{equation}\label{fm}
f''(y)\asymp|h|t^4X^{1-2c}
\end{equation}
for
\begin{equation*}
y\in\left(\frac{(X/2)^c+1}{t^2},\,\frac{X^c+2}{t^2}\right]\,.
\end{equation*}
Bearing in mind \eqref{H}, \eqref{fm} and Lemma \ref{Korput} we get
\begin{equation}\label{Hest}
H(t,h)\ll|h|^{1/2}X^{1/2}+|h|^{-1/2}t^{-2}X^{c-1/2}\,.
\end{equation}
Taking
\begin{equation}\label{M}
M=\frac{X^{\frac{2c-1}{4}}\log X}{dt}\,,
\end{equation}
by \eqref{z}, \eqref{Omega12est2}, \eqref{Hest} and Lemma \ref{tauest} we obtain
\begin{align}\label{Omega12est3}
\Omega_1,\,\Omega_2&\ll\sum\limits_{dt\leq z}\frac{X^c}{d^2t^2}\frac{\log M}{M}
+\sum\limits_{dt\leq z}\frac{\log M}{M}
\sum\limits_{h\leq M}\big(h^{1/2}X^{1/2}+h^{-1/2}t^{-2}X^{c-1/2}\big)\nonumber\\
&+\sum\limits_{dt\leq z}M\sum\limits_{h>M}\big(h^{-3/2}X^{1/2}+h^{-5/2}t^{-2}X^{c-1/2}\big)\nonumber\\
&\ll X^\varepsilon\sum\limits_{dt\leq z}\left(\frac{X^c}{Md^2t^2}
+M^{1/2}X^{1/2}+M^{-1/2}t^{-2}X^{c-1/2}\right)\nonumber\\
&\ll X^{\frac{2c+1}{4}+\varepsilon}\sum\limits_{dt\leq z}(dt)^{-1}
+X^{\frac{2c+3}{8}+\varepsilon}\sum\limits_{dt\leq z}(dt)^{-\frac{1}{2}}
+X^{\frac{6c-3}{8}+\varepsilon}\sum\limits_{dt\leq z}(dt)^{\frac{1}{2}}\nonumber\\
&=X^{\frac{2c+1}{4}+\varepsilon}\sum\limits_{n\leq z}\tau(n)n^{-1}
+X^{\frac{2c+3}{8}+\varepsilon}\sum\limits_{n\leq z}\tau(n)n^{-\frac{1}{2}}
+X^{\frac{6c-3}{8}+\varepsilon}\sum\limits_{n\leq z}\tau(n)n^{\frac{1}{2}}\nonumber\\
&\ll X^{\frac{2c+1}{4}+\varepsilon}\sum\limits_{n\leq z}n^{-1}
+X^{\frac{2c+3}{8}+\varepsilon}\sum\limits_{n\leq z}n^{-\frac{1}{2}}
+X^{\frac{6c-3}{8}+\varepsilon}\sum\limits_{n\leq z}n^{\frac{1}{2}}\nonumber\\
&\ll X^{\frac{2c+1}{4}+\varepsilon}
+X^{\frac{2c+3}{8}+\varepsilon}z^{\frac{1}{2}}
+X^{\frac{6c-3}{8}+\varepsilon}z^{\frac{3}{2}}\nonumber\\
&\ll X^{\frac{6c+1}{8}+\varepsilon}\,.
\end{align}
From \eqref{Sc4est1} and \eqref{Omega12est3} it follows
\begin{equation}\label{Sc4est2}
S^{(4)}_c(X)\ll\sum\limits_{dt\leq z}\sum\limits_{h\leq M}\frac{1}{h}
|S^{(5)}_c(X)|+X^{\frac{6c+1}{8}+\varepsilon}\,,
\end{equation}
where
\begin{equation}\label{Sc5}
S^{(5)}_c(X)=\sum\limits_{(X/2)^cd^{-2}<k\leq
X^cd^{-2}\atop{kd^2+1\equiv0\,(t^2)}}\omega(k,d,h)\,.
\end{equation}
By \eqref{omega} we have
\begin{equation}\label{omegatheta}
\omega(k,d,h)=\theta_h(kd^2)e(-h(kd^2)^\gamma)\,,
\end{equation}
where
\begin{equation}\label{theta}
\theta_h(t)=e(h(t^\gamma-(t+1)^\gamma))-1\,.
\end{equation}
Using \eqref{Sc5}, \eqref{omegatheta} and Abel's transformation we find
\begin{align}\label{Sc5est1}
S^{(5)}_c(X)&=\sum\limits_{(X/2)^cd^{-2}<k\leq X^cd^{-2}\atop{kd^2+1\equiv0\,(t^2)}}
\theta_h(kd^2)e(-h(kd^2)^\gamma)\nonumber\\
&=\theta_h(X^c)\sum\limits_{(X/2)^cd^{-2}<k\leq X^cd^{-2}\atop{kd^2+1\equiv0\,(t^2)}}e(-h(kd^2)^\gamma)\nonumber\\
&-\int\limits_{(X/2)^cd^{-2}}^{X^cd^{-2}}(\theta_h(yd^2))'\sum\limits_{(X/2)^cd^{-2}<k\leq y\atop{kd^2+1\equiv0\,(t^2)}}e(-h(kd^2)^\gamma)dy\nonumber\\
&\ll\Bigg(|\theta_h(X^c)|+d^2\int\limits_{(X/2)^cd^{-2}}^{X^cd^{-2}}|\theta'_h(yd^2)|dy\Bigg)
\max\limits_{y\in[(X/2)^cd^{-2},X^cd^{-2}]}|S^{(6)}_c(X,y)|\,,
\end{align}
where
\begin{equation}\label{Sc6}
S^{(6)}_c(X,y)=\sum\limits_{(X/2)^cd^{-2}<k\leq y\atop{kd^2+1\equiv0\,(t^2)}}e(h(kd^2)^\gamma)\,.
\end{equation}

Consider $\theta_h(X^c)$. By \eqref{theta} and the mean-value theorem we get
\begin{equation}\label{thetaest1}
|\theta_h(X^c)|\leq2|\sin\big(\pi h((X^c)^\gamma-(X^c+1)^\gamma)\big)|\ll
|h|\big|(X^c)^\gamma-(X^c+1)^\gamma\big|\ll|h|X^{1-c}\,.
\end{equation}
On the other hand
\begin{equation*}
\theta'_h(t)=2\pi i h\gamma\big(t^{\gamma-1}-(t+1)^{\gamma-1}\big)e\big(h(t^\gamma-(t+1)^\gamma)\big)
\ll|h|t^{\gamma-2}
\end{equation*}
therefore
\begin{equation}\label{thetaest2}
\theta'_h(t)\ll|h|X^{1-2c}\quad\mbox{ for }\quad t\in[(X/2)^c,X^c] \,.
\end{equation}
Bearing in mind \eqref{Sc5est1} -- \eqref{thetaest2} we obtain
\begin{equation}\label{Sc5est2}
S^{(5)}_c(X)\ll|h|X^{1-c}\max\limits_{y\in[(X/2)^cd^{-2},X^cd^{-2}]}|S^{(6)}_c(X,y)|\,.
\end{equation}
Thus from \eqref{Sc4est2}, \eqref{Sc6} and \eqref{Sc5est2} it follows
\begin{equation}\label{Sc4est3}
S^{(4)}_c(X)\ll X^{1-c}\sum\limits_{dt\leq z}\sum\limits_{h\leq M}
\max\limits_{y\in[(X/2)^cd^{-2},X^cd^{-2}]}|S^{(6)}_c(X,y)|+X^{\frac{6c+1}{8}+\varepsilon}\,.
\end{equation}
By \eqref{Sc6} we have
\begin{equation*}
S^{(6)}_c(X,y)=\sum\limits_{((X/2)^c+1)t^{-2}<l\leq(yd^2+2)t^{-2}}e(h(lt^2-1)^\gamma)
\end{equation*}
and arguing as in \eqref{H} we find
\begin{equation}\label{Sc6est}
\max\limits_{y\in[(X/2)^cd^{-2},X^cd^{-2}]}|S^{(6)}_c(X,y)|
\ll|h|^{1/2}X^{1/2}+|h|^{-1/2}t^{-2}X^{c-1/2}\,.
\end{equation}
Using \eqref{z}, \eqref{M}, \eqref{Sc4est3}, \eqref{Sc6est}
and Lemma \ref{tauest} we get
\begin{align}\label{Sc4est4}
S^{(4)}_c(X)&\ll X^{1-c}\sum\limits_{dt\leq z}\sum\limits_{h\leq M}(h^{\frac{1}{2}}X^{1/2}
+h^{-1/2}t^{-2}X^{c-1/2})+X^{\frac{6c+1}{8}+\varepsilon}\nonumber\\
&\ll X^{\frac{3-2c}{2}}\sum\limits_{dt\leq z}M^{\frac{3}{2}}
+X^{1/2}\sum\limits_{dt\leq z}M^{1/2}+X^{\frac{6c+1}{8}+\varepsilon}\nonumber\\
&\ll X^{\frac{9-2c}{8}+\varepsilon}\sum\limits_{dt\leq z}(dt)^{-\frac{3}{2}}
+X^{\frac{2c+3}{8}+\varepsilon}\sum\limits_{dt\leq z}(dt)^{-\frac{1}{2}}
+X^{\frac{6c+1}{8}+\varepsilon}\nonumber\\
&=X^{\frac{9-2c}{8}+\varepsilon}\sum\limits_{n\leq z}\tau(n)n^{-\frac{3}{2}}
+X^{\frac{2c+3}{8}+\varepsilon}\sum\limits_{n\leq z}\tau(n)n^{-\frac{1}{2}}
+X^{\frac{6c+1}{8}+\varepsilon}\nonumber\\
&\ll X^{\frac{9-2c}{8}+\varepsilon}\sum\limits_{n\leq z}n^{-\frac{3}{2}}
+X^{\frac{2c+3}{8}+\varepsilon}\sum\limits_{n\leq z}n^{-\frac{1}{2}}
+X^{\frac{6c+1}{8}+\varepsilon}\nonumber\\
&\ll X^{\frac{2c+3}{8}+\varepsilon}z^{1/2}+X^{\frac{6c+1}{8}+\varepsilon}\nonumber\\
&\ll X^{\frac{6c+1}{8}+\varepsilon}\,.
\end{align}
Bearing in mind \eqref{Sc1est1}, \eqref{Sc3est2} and \eqref{Sc4est4} we obtain
\begin{equation}\label{Sc1est2}
S^{(1)}_c(X)=\frac{1}{2}\sigma X+\mathcal{O}\left(X^{\frac{6c+1}{8}+\varepsilon}\right)\,.
\end{equation}
\textbf{Estimation of} $\mathbf{S^{(2)}_c(X)}$

Using \eqref{Sc2} we write
\begin{equation}\label{Sc2est1}
S^{(2)}_c(X)\ll(\log X)^2\sum\limits_{D\leq d<2D}\sum\limits_{T\leq t<2T}
\sum\limits_{(X/2)^cd^{-2}<k\leq X^cd^{-2}\atop{kd^2+1\equiv0\,(t^2)}}1\,,
\end{equation}
where
\begin{equation}\label{DT}
\frac{1}{2}\leq D,T\leq\sqrt{X^c+1}\,,\quad DT\geq\frac{z}{4}\,.
\end{equation}
On the one hand \eqref{Sc2est1} and Lemma \ref{tauest} give us
\begin{align}\label{Sc2est2}
S^{(2)}_c(X)&\ll(\log X)^2\sum\limits_{T\leq t<2T}\sum\limits_{l\leq(X^c+1)T^{-2}}\sum\limits_{D\leq d<2D}
\sum\limits_{(X/2)^cd^{-2}<k\leq X^cd^{-2}\atop{kd^2=lt^2-1}}1\nonumber\\
&\ll(\log X)^2\sum\limits_{T\leq t<2T}\sum\limits_{l\leq(X^c+1)T^{-2}}\tau(lt^2-1)\nonumber\\
&\ll X^\varepsilon\sum\limits_{T\leq t<2T}\sum\limits_{l\leq(X^c+1)T^{-2}}1\nonumber\\
&\ll X^{c+\varepsilon}T^{-1}\,.
\end{align}
On the other hand \eqref{Sc2est1} and Lemma \ref{tauest} imply
\begin{align}\label{Sc2est3}
S^{(2)}_c(X)&\ll(\log X)^2\sum\limits_{D\leq d<2D}\sum\limits_{(X/2)^cd^{-2}<k\leq X^cd^{-2}}\sum\limits_{T\leq t<2T}
\sum\limits_{l\leq(X^c+1)T^{-2}\atop{kd^2+1=lt^2}}1\nonumber\\
&\ll(\log X)^2\sum\limits_{D\leq d<2D}\sum\limits_{k\leq X^cD^{-2}}\tau(kd^2+1)\nonumber\\
&\ll X^\varepsilon\sum\limits_{D\leq d<2D}\sum\limits_{k\leq X^cD^{-2}}1\nonumber\\
&\ll X^{c+\varepsilon}D^{-1}\,.
\end{align}
By \eqref{DT} -- \eqref{Sc2est3} it follows
\begin{equation}\label{Sc2est4}
S^{(2)}_c(X)\ll X^{c+\varepsilon}z^{-\frac{1}{2}}\,.
\end{equation}
Using \eqref{z} and \eqref{Sc2est4} we find
\begin{equation}\label{Sc2est5}
S^{(2)}_c(X)\ll X^{\frac{6c+1}{8}+\varepsilon}\,.
\end{equation}
Bearing in mind \eqref{Scest1}, \eqref{Sc1est2} and \eqref{Sc2est5} we obtain
\begin{equation}\label{Scest2}
S_c(X)=\frac{1}{2}\sigma X+\mathcal{O}\left(X^{\frac{6c+1}{8}+\varepsilon}\right)\,,
\end{equation}
where $\sigma$ is denoted by \eqref{sigma}.

The Theorem is proved.

\vskip20pt
\footnotesize
\begin{flushleft}
S. I. Dimitrov\\
Faculty of Applied Mathematics and Informatics\\
Technical University of Sofia \\
8, St.Kliment Ohridski Blvd. \\
1756 Sofia, BULGARIA\\
e-mail: sdimitrov@tu-sofia.bg\\
\end{flushleft}
\end{document}